\documentclass[11pt,a4paper,twoside,reqno]{amsart}

\usepackage{amssymb}
\usepackage{latexsym}

\usepackage{graphicx}
\usepackage{a4wide}

\usepackage{amsmath}
\usepackage{amsthm}

\usepackage{color}
\usepackage{caption}
\usepackage{subcaption}
\usepackage{hyperref}
\usepackage{multirow}
\usepackage{booktabs}

\newtheorem{theorem}{Theorem}
\newtheorem{definition}{Definition}
\newtheorem{lemma}{Lemma}
\newtheorem{remark}{Remark}

\newcommand{\norm}[1]{\left\| #1 \right\|}

\newcommand{\pt}{\partial_t}
\newcommand{\ptt}{\partial^2_{tt}}

\newcommand{\bx}{\bar{x}}
\newcommand{\bt}{\bar{t}}
\newcommand{\tx}{\tilde{x}}

\renewcommand{\i}{\ifmmode\mathit{\mathchar"7010 }\else\char"10 \fi}
\renewcommand{\j}{\ifmmode\mathit{\mathchar"7011 }\else\char"11 \fi}
\newcommand{\R}{\mathbb{R}}
\newcommand{\N}{\mathbb{N}}

\newcommand{\vfi}{\varphi}

\begin{document}\large

\title[A fast-convolution based space-time Chebyshev spectral method for peridynamic models]{A fast-convolution based space-time Chebyshev spectral method for peridynamic models}

\author[L. Lopez]{Luciano Lopez}

\author[S. F. Pellegrino]{Sabrina Francesca Pellegrino
}

\address[L. Lopez]{Dipartimento di Matematica, Universit\`a degli Studi di Bari Aldo Moro, via E. Orabona 4, 70125 Bari, Italy}
\email{luciano.lopez@uniba.it}

\address[S. F. Pellegrino]{Dipartimento di Management, Finanza e Tecnologia, Universit\`a LUM Giuseppe Degennaro, S.S. 100 Km 18 - 70010 Casamassima (BA), Italy}
\email{pellegrino@lum.it}

\begin{abstract}

Peridynamics is a nonlocal generalization of continuum mechanics theory which adresses discontinuous problems without using partial derivatives and replacing its by an integral operator. As a consequence, it finds applications in the framework of the development and evolution of fractures and damages in elastic materials.

In this paper we consider a one-dimensional nonlinear model of peridynamics and propose a suitable two-dimensional fast-convolution spectral method based on Chebyshev polynomials to solve the model. This choice allows us to gain the same accuracy both in space and time. We show the convergence of the method and perform several simulations to study the performance of the spectral scheme.


\end{abstract}

\maketitle

{\bf{\textit{Keywords.}}} nonlinear peridynamics, Chebyshev spectral methods, Chebyshev polynomials, convolution product, Fast Fourier Transform.




\section{Introduction}
\label{sec:intro}

In the framework of continuum mechanic theory, peridynamics is a nonlocal version of elasticity model able to describe the formation and the evolution of fractures and damages in elastic materials. It was introduced by Silling in~\cite{Silling_2000} and consists in a second order in time partial integro-differential equation.

The main capability of the model is that it avoids the use of partial derivatives in space, so it can address discontinuous problems~\cite{Delia2017,DSDPR,DELIA2021,Pellegrino,alebrahim,Alebrahim2021,acoclite,CoclitePellegrino}.

Most standard approaches used to approximate the solution of the peridynamic equation make use of meshfree methods with the Gauss two points quadrature (see~\cite{Emmrich_Weckner_2007,berardidifonzo}) and finite difference methods with St\"ormer-Verlet scheme (see~\cite{Galvanetto2018}). 

Both methods need $\mathcal{O}(N^2)$ cost per time step and perform well when the nonlocality covers a small portion of the domain (see~\cite{Madenci2010}). In particular, in~\cite{CFLMP,Pellegrino2020}, the authors make a survey of the most implemented numerical methods in the peridynamic framework and propose a different approach based on spectral techniques.

Spectral methods are an important tool for the numerical solution of many applied problems, as they allow to have high-order accuracy for smooth problems. They consist in reformulate the original problem in the frequency space, by decomposing the solution as linear combination of a suitable basis. 

In~\cite{LP,LP2021,Jafarzadeh,Bobaru2021,Jafarzadeh2021} the authors propose a spectral discretization of the model based on the Fourier trigonometric polynomials. And, since this approach can be applied only to periodic problems, the authors perform a volume penalization technique to extend the method to the non periodic setting.

A different way to overcome the limitation of periodic solution consists in replacing the Fourier polynomials by the Chebyshev polynomials (see~\cite{LPcheby}). This method results to be very efficient in terms of computational cost per time step as it can benefit of the use of the Fast Fourier Transform algorithm (FFT).

For time integration St\"ormer-Verlet scheme or Newmark-$\beta$ method are commonly used in the context of wave propagation or peridynamics (see for instance~\cite{LP,LPcheby,coclites}).

In this work, we propose a fast-convolution fully spectral method in space and time based on the implementation of Chebyshev polynomials, in order to have the same accuracy in both variables. The basic idea is to study the problem in the two-dimensional Cartesian $(x,t)$ bounded space-time domain and to expand the unknown function in Chebyshev polynomials for the spatial variable $x$ as well as for the time variable $t$. One of the advantages of this approach is that we do not need to integrate in time the semi-discrete method. Indeed, this step is substituted by a numerial procedure to solve an algebraic system. Additionally, the choice of using Chebyshev polynomials releases us from the use of periodic bounadry conditions.

The paper is organized as follows: in Section~\ref{sec:state} we describe the non linear peridynamic model we aim to study. Useful properties of Chebyshev polynomials are summarized in Section~\ref{sec:cheby}. In Section~\ref{sec:method} we construct the fully spectral method to solve the peridynamic equation and prove the convergence of the proposed method. Simulations and results are shown in Section~\ref{sec:test}. Finally, Section~\ref{sec:concl} concludes the paper.

\section{Statement of the problem}
\label{sec:state}

Peridynamics is a non local version of continuum mechanics based on long-range interactions. The main motivation of the development of such theory relies on the necessity to find an analytical description of discontinuous phenomena like fractures and cracks. The long-range interactions are parametrized thanks to the introduction of a scalar quantity $\delta>0$, called horizon, as a measure of the non locality.

Let $\Omega\subset \R$ be the spatial domain. We consider the following non linear peridynamic model
\begin{equation}
\ptt u(x,t)=\int_{B_{\delta}(x)} f\left(x-x',\left(u(x',t)-u(x,t)\right),t\right)\,dx',\qquad x\in\Omega,\,t>0
\label{eq:nonlinperid}
\end{equation}
which describes the evolution of a material body, and where the unknown $u$ represents the displacement field.
Neglecting the dependence on time, the pairwise force function $f$ is supposed to decompose as follows
\begin{equation}
f(\xi,\eta)=C(\xi)H(\eta),\qquad\text{for every }\,(\xi,\eta)\in\Omega\times\Omega,
\label{eq:f}
\end{equation}
where the function $C$ is an even scalar function, called {\em micromodulus function} (see~\cite{WECKNER2005705}), which vanishes when $|\xi|>\delta$ and in what follows we assume $C\in L^{\infty}(\R)$, while $H$ is an odd function globally Lipschitz continuous (see~\cite{Emmrich_Puhst_2015}), namely there is a nonnegative function $\ell\in L^1(B_{\delta}(0)\cap L^{\infty}(B_{\delta}(0)))$ such that for all $\xi\in\R$, with $|\xi|\le\delta$ and $\eta,\,\eta'$ there holds
\begin{equation}
\label{eq:lipschitz}
\left|H(\eta')-H(\eta)\right|\le \ell(\xi)\left|\eta'-\eta\right|.
\end{equation}

In this work, we will work with $H(\eta)=\eta^3$. This kind of power-type non-linearity appears to be very useful from a numerical point of view, as it allows us to take advantage of the properties of Chebyshev transform and convolution products (see for instance~\cite{LPcheby,LPeigenv}). Moreover, from an analytical point of view, this choice is justified by the fact that peridynamic integral operator resembles a fractional derivative (see for instance~\cite{garrappapopolizio}) and in this setting a well-posedness of the model is achieved (see~\cite{Emmrich_Puhst_2013,CDMV}).

We denote by $\mathcal{L}$ the peridynamic integral operator of~\eqref{eq:nonlinperid}, and thanks to the assumption~\eqref{eq:f} on $f$, we can write it in the following way:
\begin{align}
\label{eq:L}
\mathcal{L} u(x,t) &= \int_\R C(x-x')\ u^3(x',t)\,dx' -3\ u(x,t) \int_\R C(x-x')u^2(x',t)\,dx'\\
&\quad +3\ u^2(x,t)\int_\R C(x'-x)\ u(x',t)\,dx' -u^3(x,t)\int_\R C(x-x')\,dx' \notag\\
&=\left(C\ast u^3\right)(x,t)-3\ u(x,t)\left(C\ast u^2\right)(x,t)+3\ u^2(x,t)\left(C\ast u\right)(x,t)\notag\\
&\quad -\beta\ u^3(x,t),\notag 
\end{align}
where $\ast$ denotes the convolution product and
\[
\beta=\int_\R C(x)\,dx.
\]

Therefore, the peridynamic equation becomes
\begin{equation}
\label{eq:perideq}
\ptt u(x,t) = \mathcal{L}\ u(x,t).
\end{equation}

We add the initial condition 
\[
u(x,0)=u_0(x),\qquad \pt u(x,0)=v_0(x),
\]
so the peridynamic model we aim to study is
\begin{equation}
\label{eq:model}
\begin{cases}
\ptt u(x,t) = \mathcal{L}\ u(x,t)\qquad&x\in\Omega,\,t>0\\
u(x,0)=u_0(x),\quad \pt u(x,0)=v_0(x)\qquad & x\in\Omega
\end{cases}
\end{equation}

Since the peridynamic operator $\mathcal{L}$ is decomposed as the sum of convulution products, the spectral approach to discretized it represents a good framework in order to have good accuracy in the solution and to reduce the computational cost.

In particular, we can discretize the model~\eqref{eq:model} by exploiting Chebyshev polynomials' theory.

In next section, we briefly recall the definition of Chebyshev polynomials and their main properties in relation with the convolution products and the differential operators.  

\section{Basic properties of Chebyshev polynomials}
\label{sec:cheby}

This section is devoted to provide an overview on Chebyshev polynomials and Chebyshev collocation method.

Chebyshev polynomials of the first kind, $T_n(x)$ are explicitly defined as
\begin{equation}
\label{eq:Tn}
T_n(x)=\cos(n\arccos(x)),\qquad x\in[-1,1],\quad n\in\N.
\end{equation}
Without loss of generality, it is always possible to introduce a new variable $y\in[a,b]$ and a linear map which allows to scaling the polynomials from $[-1,1]$ to $[a,b]$,
so, we restrict our discussion to the normalized domain $[-1,1]$.

Chebyshev polynomials~\eqref{eq:Tn} are orthogonal with respect the weight function $w(x)=(\sqrt{1-x^2})^{-1}$, and their boundary values as well as those of their first and second derivatives are given by
\begin{align*}
T_n(\pm 1)&= (\pm 1)^n\\
T_n'(\pm 1)&=(\pm 1)^{n+1} n^2\\
T_n''(\pm 1)&=(\pm 1)^n n^2\frac{n^2-1}{3}
\end{align*}

Moreover, Chebyshev polynomials have an interpolation property: any sufficiently smooth function $f$ defined on the interval $[-1,1]$ can be expanded in a series of Chebyshev polynomials. The $(N+1)$-terms interpolation of $f$ is denoted by $f^N$ and has the following expression
\begin{equation}
\label{eq:chebser}
f^N(x)=\sum_{n=0}^N f_n T_n(x),
\end{equation} 
where $f_n$ are the coefficients of the expansion, whose discrete expression depends on the choice of collocation points.

If we fix a grid corresponding to the Gauss-Lobatto collocation points
\begin{equation}
\label{eq:GL}
x_k=\cos\left(\frac{k\pi}{N}\right),\qquad k=0,\dots,N,
\end{equation}
we can express the Chebyshev coefficients $f_n$ as follows
\begin{equation}
\label{eq:chebcoef}
f_n=\frac{1}{\gamma_n}\sum_{k=0}^N f(x_k) T_n(x_k)w_k,
\end{equation}
where
\[
w_k=\begin{cases}
\frac{\pi}{2N}\qquad&k=0,\,N\\
\frac{\pi}{N}\qquad& k=1,\dots,N-1
\end{cases}
\]
and the normalization constant $\gamma_n$ is given by
\[
\gamma_n=\begin{cases}
\pi\qquad&n=0,\,N\\
\frac{\pi}{2}\qquad &n=1,\dots,N-1
\end{cases}
\]
The choice of the Gauss-Lobatto points as grid points for the discretization is very useful as it can avoid the Gibb's phenomenon at the boundaries.

Thanks to their definition, Chebyshev polynomials are strictly related to the trigonometric cosine functions and, as a consequence, the finite series~\eqref{eq:chebser} can be efficiently computed by a Fourier cosine transform by using a Fast Fourier Transform (FFT) algorithm.

To solve discretized problems, we need to look for the relation between the Chebyshev coefficients and the derivative of any order to the coefficient of the interpolant function itself.

Let $f$ be a sufficiently smooth function approximated by $f^N$ defined in~\eqref{eq:chebser}, where the coefficients $f_n$, $n=0,\dots,N$, are given by~\eqref{eq:chebcoef}. Then, the coefficients of its first derivative $f'$ are given by
\begin{equation}
\label{eq:f'coef}
f_n'=\frac{2}{c_n}\sum_{\stackrel{k=n+1}{k+n\,\, \text{odd}}}^N k\ f_k
\end{equation}
for 
\[
c_n=\begin{cases}
2\qquad&n=0\\
1\qquad&\text{otehrwise}
\end{cases}
\]
or equivalently the coefficients $f_n'$ can be computed as a matrix multiplication:
\begin{equation}
\label{eq:f'coef2}
f_n'=\sum_{k=0}^N D_{nk}\ f_k,
\end{equation}
where $D=(D_{nk})$ is a $(N+1)\times (N+1)$ derivative matrix, with the following representation
\begin{equation}
\label{eq:D}
D=\begin{pmatrix}
0&1&0&3&0&\cdots&n&\cdots&N\\
0&0&4&0&8&&0&\cdots&0\\
0&0&0&6&0&&2n&\cdots&2N\\
0&0&0&0&8&&0&\cdots&0\\
\vdots&0&0&0&0&\ddots&\vdots&&\vdots\\
\vdots&0&0&0&0&\ddots&2n&&0\\
\vdots&&&&&\cdots&0&&2N\\
0&&&&&\cdots&&&0
\end{pmatrix}
\end{equation}
It is an upper triangular matrix with its main diagonal terms equal to zero.

In an analogous way, we can find the coefficients for higher order derivatives, by taking the power of the matrix $D$:
\[
f_n^{\ell}=\sum_{k=0}^N (D^\ell)_{nk} f_k
\]

Let denote by $\mathcal{F}_N$ the linear map which associates to a function $f$ its Chebyshev discrete coefficients $f_n$, $n=0,\dots,N$, defined in~\eqref{eq:chebcoef}, and let $\mathcal{F}_N^{-1}$ be its inverse discrete transform defined by~\eqref{eq:chebser}.

We have that $\mathcal{F}_N$ satisfies the following property when its is composed with a differential operator:
\begin{equation}
\label{eq:diffch}
\frac{d^k f}{d\xi^k}(\xi)=\mathcal{F}_N^{-1}\left((-\Im \xi)^k\mathcal{F}_N(f)\right),
\end{equation}
where $\Im$ denotes the imaginary unit such that $\Im^2=-1$. This property is equivalent to using the derivative matrix $D$ defined in~\eqref{eq:D}.

 Moreover, when $\mathcal{F}_N$ is applied to a convolution product we find
\begin{equation}
\label{eq:chebconv}
f\ast g = \mathcal{F}_N^{-1}\left(\mathcal{F}_N(f)\mathcal{F}_N(g)\right).
\end{equation}
In~\cite{baszenski}, the authors show the relation between the Chebyshev coefficients of $(f\ast g)$ in terms  of the Chebyshev coefficients of $f$ and $g$.

\begin{remark}
When we deal with functions depending both on the space and the time variable, let say $f(x,t)$, we can still approximate it by finite Chebyshev series expansion in both space and time. In this context, we seek an approximation function $f^N(x,t)$ in the two variables $(x,t)\in[-1,1]^2$ such that
\begin{equation}
\label{eq:chebser2d}
f^N(x,t)=\sum_{j=0}^{N_x}\sum_{k=0}^{N_t} f_{jk}\ T_j(x)\ T_k(t)
\end{equation}
where $N_x$ and $N_t$ represents the total number of collocation points in space and time, respectively. The coefficients $f_{jk}$, $j=0,\dots,N_x$
, $k=0,\dots,N_t$, are the Chebyshev coefficients of the discrete Chebyshev expansion and when the grid points are the Gauss-Lobatto points $(x_n,t_m)=(\cos(n\pi/N_x),\cos(m\pi/N_t))$, then their expression is given by
\begin{equation}
\label{eq:fcoef}
f_{jk}=\frac{1}{\gamma_j\gamma_k}\sum_{n=0}^{N_x}\sum_{m=0}^{N_t} f(x_n,t_m)\ T_j(x_n)T_k(t_m)w_n w_m
\end{equation}

For the purpose of our work, we mention here only the expansion of the second order derivative in time:
\begin{equation}
\label{eq:pttchebser}
\ptt\ f^N(x,t)=\sum_{j=0}^{N_x}\sum_{k=0}^{N_t}\sum_{\ell=0}^{N_t} \hat{D}_{k\ell}^{(t)}\ f_{j\ell}\ T_j(x)\ T_k(t) 
\end{equation}
where $\hat{D}=D\cdot D$, with $\hat{D}=(\hat{D}_{k\ell})=\sum_{j=0}^N D_{kj}D_{j\ell}$, and the superscript $(t)$ in the derivative matrix $D$ denotes the differentiation with respect to the temporal coordinates.

We can compact the expression~\eqref{eq:pttchebser} as follows
\begin{equation}
\label{eq:pttcheb}
\ptt\ f^N(x,t)=\sum_{j=0}^{N_x}\sum_{k=0}^{N_t} \hat{f}_{kj}\ T_j(x)\ T_k(t),
\end{equation}
where 
\begin{equation}
\label{eq:a}
\hat{f}_{kj}=\sum_{\ell=0}^{N_t} \hat{D}_{k\ell}^{(t)}\ f_{j\ell}.
\end{equation}

Even in this case we can benefit of the implementation of the Fast Fourier Transform algorithm in the two-dimensional setting to compute the coefficients $u_{jk}$ in~\eqref{eq:chebser2d}.

Additionally, the same results as in~\eqref{eq:diffch} and in~\eqref{eq:chebconv} hold in the bi-dimensional case.
\end{remark}

\section{Chebyshev spectral method for the fully discrete problem}
\label{sec:method}

We develop a fast-convolution fully spectral method to solve the non-linear peridynamic problem~\eqref{eq:model}.

Without loss of generality, we assume $\Omega=[-1,1]$ and $t\in[-1,1]$, we fix $N+1>0$ as the total number of collocation points in both space and time direction, and we take the Gauss-Lobatto points $(x_n,t_m)$ as grid points for the discretization.

We look for an approximation of $u(x,t)$ in the form
\begin{equation}
\label{eq:approx}
u^N(x,t)=\sum_{j=0}^N \sum_{k=0}^N u_{jk}\ T_j(x)\ T_k(t).
\end{equation}

Substituting $u^N(x,t)$ into~\eqref{eq:L}, we find the full expression of the peridynamic operator
\begin{align}
\label{eq:Ldiscr}
\mathcal{L}\ u^N(x,t) &=\left(C\ast \left(u^N\right)^3\right)(x,t)-3\ u^N(x,t) \left(C\ast \left(u^N\right)^2\right)(x,t)\\
&\quad + 3\ (u^N)^2(x,t)(C\ast u^N(x,t))-\beta \left(u^N\right)^3(x,t)\notag
\end{align}

If we evaluate $u^N(x,t)$ at $(x_n,t_m)$, we obtain the discrete form of~\eqref{eq:Ldiscr}
\begin{align}
\label{eq:L-discret}
\mathcal{L}\ u^N_{nm} =& \frac{2}{N}\left(\mathcal{F}_N^{-1}\left(\mathcal{F}_N(C)\mathcal{F}_N\left( \left(u^N_{nm}\right)^3\right)\right)\right) \\ 
& -\frac{6}{N} \left(\mathcal{F}_N^{-1}\left(\mathcal{F}_N\left( \left(u^N_{nm}\right)\right)\ast\left(\mathcal{F}_N( C) \mathcal{F}_N\left( \left(u^N_{nm}\right)^{2}\right) \right)\right)\right)\notag \\
& +\frac{6}{N} \left(\mathcal{F}_N^{-1}\left(\mathcal{F}_N\left( \left(u^N_{nm}\right)^2\right)\ast\left(\mathcal{F}_N( C) \mathcal{F}_N\left( \left(u^N_{nm}\right)\right) \right)\right)\right)\notag \\
& - \beta  ( u^N_{nm})^3,\notag
\end{align}
where $u_{nm}^N$ approximates $u^N(x_n,t_m)$.

Moreover, thanks to the differentiation theorem of the Chebyshev transform, we have
\begin{equation}
\label{eq:d2t}
\ptt u^N(x_n,t_m)=\frac{2}{N}\mathcal{F}_N^{-1}\left( t_m^2 \mathcal{F}_N(u^N_{nm})\right),
\end{equation}
or equivalently
\begin{equation}
\label{eq:d2tt}
\ptt u^N(x_n,t_m)=\sum_{j=0}^N\sum_{k=0}^N \hat{u}_{kj}\ T_j(x_n)T_k(t_m),
\end{equation}
with $\hat{u}_{kj}$ as in~\eqref{eq:a}, for $k,\,j=0,\dots,N$.

Thus, we can consider the discrete form of the model~\eqref{eq:model}
\begin{equation}
\label{eq:modeldiscret}
\begin{cases}
\sum_{j=0}^N\sum_{k=0}^N \hat{u}_{kj}\ T_j(x_n)T_k(t_m) - \mathcal{L}(u^N_{nm},t_m)=0,\qquad n=0,\dots,N,\,m=1,\dots,N\\
\sum_{j=0}^N\sum_{k=0}^N (-1)^k u_{jk}\ T_j(x_n)=u_0(x_n),\qquad n=0,\dots,N\\
\sum_{j=0}^N\sum_{k=0}^N\sum_{\ell=0}^N (-1)^k D_{k\ell}  u_{j\ell}\ T_j(x_n)=v(x_n),\qquad n=0,\dots,N,
\end{cases}
\end{equation}
where in the peridynamic operator $\mathcal{L}$ we have shown explicitly the time dependence.

Solving the above non-linear system with respect to $u^N_{nm}$, we find an approximated solution of~\eqref{eq:model} having the form as in~\eqref{eq:approx}. In practice, in the next section, we use the FSOLVE command implemented in MATLAB software to solve the system~\eqref{eq:modeldiscret}. It consists in a quasi-Newton method called Levenberg-Marquardt method.

We analyze the convergence of the proposed method in the space of functions which admit a modulus of continuity. We start by giving some definitions and recalling some standard results.

\begin{definition}
A continuous function $W:\R_+\to\R_+$ is called {\em modulus of continuity} if it satisfies the following properties:
\begin{itemize}
\item $W$ is increasing,
\item $\lim_{z\to 0} W(z)=0$,
\item $W(z_1+z_2)\le W(z_1)+W(z_2)$, for $z_1$, $z_2\in\R_+$,
\item there exists a constant $c>0$, such that $z\le c\ W(z)$, for all $0<z\le 2$.
\end{itemize}
\end{definition}
An example of a modulus of continuity is given by the functions $W(z)=z^\alpha$, $0<\alpha\le 1$.

Let $B^2$ be the unit ball in $\R^2$.
\begin{definition}
\label{def:2}
We say that a continuous function $u(\cdot,\cdot)$ on $B^2$ {\em admits a modulus of continuity} $W(\cdot)$ if
\begin{equation}
\label{eq:modcont}
|u(\cdot,\cdot)|_W=\sup_{(\bar{x},\bar{t})\ne(\tilde{x},\tilde{t})}\ \left\{\frac{u(\bar{x},\bar{t})-u(\tilde{x},\tilde{t})}{W(|||(\bar{x},\bar{t})-(\tilde{x},\tilde{t})|||)},\ (\bar{x},\bt),\,(\tilde{x},\tilde{t})\in B^2\right\}
\end{equation}
is finite.
\end{definition}
In~\eqref{eq:modcont}, $|||(\bar{x},\bar{t})-(\tilde{x},\tilde{t})|||=\max\{|\bt-\tilde{t}|,|\bx-\tx|\ :\ (\bar{x},\bar{t}),\ (\tilde{x},\tilde{t})\in\bar{\Omega}, (\bar{x},\bar{t})\ne(\tilde{x},\tilde{t})\}$.

We denote the class of all functions satisfying Definition~\ref{def:2} by $\mathcal{C}^{0}_W(B^2)$. Then, it is a Banach space with the norm
\begin{equation}
\label{eq:normC0W}
\norm{u(\cdot,\cdot)}_{0,W}=\norm{u(\cdot,\cdot)}_{\infty} + |u(\cdot,\cdot)|_W.
\end{equation}
Moreover, we denote the class of $k$-times differentiable functions on $B^2$ whose $k$-th derivatives admit $W$ as a modulus of continuity by $\mathcal{C}^{k}_W$. It is a Banach space under the norm
\begin{equation}
\label{eq:normCkW}
\norm{u(\cdot,\cdot)}_{k,W}=\sum_{|s|\le k} \norm{\frac{\partial^s u}{\partial t^s}}_\infty + \sum_{|s|\le k} \norm{\frac{\partial^s u}{\partial x^s}}_\infty + \sum_{|s|= k} \left|\frac{\partial^s u}{\partial t^s}\right|_W + \sum_{|s|= k} \left|\frac{\partial^s u}{\partial x^s}\right|_W.
\end{equation}

We can extend the previous definition on $\bar{\Omega}=[-1,1]\times[-1,1]$ as follows
\begin{align}
\label{eq:CkWO}
\mathcal{C}^k_W(\bar{\Omega})=\Biggl\{u\in\mathcal{C}^k(\bar{\Omega})\,:\,&\text{for each $(\tx,\tilde{t})\in\bar{\Omega}$ there exists a map $\phi:B^2\to\bar{\Omega}$}\\
&\text{such that $(\tx,\tilde{t})\in int(\phi(B^2))$ and $f\circ\phi\in \mathcal{C}^k_W(B^2)$}\Biggr\}\notag
\end{align}

Since the multiplication by a $\mathcal{C}^\infty$ function and the composition with a $\mathcal{C}^\infty$ function are continuous linear transformations, it is possible to show that if
\[
\phi_i:B^2\to\bar{\Omega},\qquad i=1,\dots,\ell,
\]
are a finite collection of maps with
\[
\bar{\Omega}=\bigcup_{i=1}^\ell int(\phi_i(B^2)),
\]
then $u(\cdot,\cdot)\in\mathcal{C}^k_W(\bar{\Omega})$ if and only if $(u\circ\phi)(\cdot,\cdot)\in\mathcal{C}^k_W(B^2)$ for each $i=1,\dots,\ell$. Moreover, the space $\mathcal{C}^k_W(\bar{\Omega})$ is a Banach space under the norm
\begin{equation}
\label{eq:normCkWO}
\norm{u(\cdot,\cdot)}_{k,W}=\sum_{i=1}^{\ell} \norm{(u\circ \phi_i)(\cdot,\cdot)}_{k,W}.
\end{equation}
Additionally, any other choice of finitely many maps covering $\bar{\Omega}$ provides an equivalent norm for the Banach space (for more details see~\cite{ragozin}). 

Let $\mathcal{P}(N,N,\bar{\Omega})$ be the space of all polynomials of total degree at most $2N$ on $\bar{\Omega}$, namely
\[
\mathcal{P}(N,N,\bar{\Omega})=\left\{p(\tx,\tilde{t})=\sum_{i=0}^N\sum_{j=0}^N b_{ij}\tx^i\tilde{t}^j\,:\,(\tx,\tilde{t})\in\bar{\Omega},\,b_{ij}\in\R\right\}.
\]
The following result is a generalization of the Stone-Weierstrass theorem on the space $\mathcal{C}^k_W(\bar{\Omega})$.
\begin{theorem}[see~\cite{ragozin}]
\label{th:ST}
For any $u(\cdot,\cdot)\in\mathcal{C}^k_W(\bar{\Omega})$, there exist a polynomial $p(\cdot,\cdot)\in\mathcal{P}(N,N,\bar{\Omega})$ such that
\begin{equation}
\label{eq:ST}
\norm{u(\cdot,\cdot)-p(\cdot,\cdot)}_\infty\le\frac{L_0\ L_1}{(2N)^k}\ W\left(\frac{1}{(2N)^k}\right),
\end{equation}
where $L_1=\norm{u(\cdot,\cdot)}_{k,W}$ and $L_0$ is a constant depending on $W$, but independent of $N$.
\end{theorem}

In order to prove the convergence of the method and the existence of solutions of the system~\eqref{eq:modeldiscret}, we reformulate it as a system of algebraic inequalities in the following way
\begin{equation}
\label{eq:dismodel}
\begin{cases}
\left|\sum_{j=0}^N\sum_{k=0}^N \hat{u}_{kj}\ T_j(x_n)T_k(t_m) - \mathcal{L}(u^N_{nm}, t_m)\right|\le \frac{\sqrt{N}}{(2N-2)^2}\ W\left(\frac{1}{(2N-2)^2}\right),&\quad n=0,\dots,N,\,m=1,\dots,N,\\
\left|\sum_{j=0}^N\sum_{k=0}^N (-1)^k u_{jk}\ T_j(x_n)-u_0(x_n)\right|\le\frac{\sqrt{N}}{(2N-2)^2}\ W\left(\frac{1}{(2N-2)^2}\right),&\quad n=0,\dots,N,\\
\left|\sum_{j=0}^N\sum_{k=0}^N\sum_{\ell=0}^N (-1)^k D_{k\ell}  u_{j\ell}\ T_j(x_n)-v(x_n)\right|\le0,&\quad n=0,\dots,N,
\end{cases}
\end{equation}
where $N$ is sufficiently large and $W$ is a given modulus of continuity.

We can notice that
\[\lim_{N\to\infty}\frac{\sqrt{N}}{(2N-2)^2}\ W\left(\frac{1}{(2N-2)^2}\right)=0,\]
so, any solution $\bar{u}^N=(u^N_{nm})$ for $n,\,m=0,\dots,N$ of the system~\eqref{eq:dismodel} is a solution of the system~\eqref{eq:modeldiscret} when $N$ goes to infinity. As a consequence, to prove the existence of solutions of~\eqref{eq:modeldiscret}, it is sufficient to prove the existence of solutions for the system~\eqref{eq:dismodel}.

The following lemmas are preliminary to the convergence theorem.
\begin{lemma}
\label{lm:1}
Let $u\in\mathcal{C}^2_W(\bar{\Omega})$ be a solution of the peridynamic model~\eqref{eq:model}. Then there exists a function $\tilde{u}$ such that
\begin{equation}
\label{eq:est0}
\left|u(\bx,\bt)-\tilde{u}(\bx,\bt)\right|\le\frac{2L}{(2N-2)^2}\ W\left(\frac{1}{(2N-2)^2}\right),\quad (\bx,\bt)\in\bar{\Omega}
\end{equation}
for some constant $L>0$.
\end{lemma}

\begin{proof}
By Theorem~\ref{th:ST}, there exists $p(\cdot,\cdot)\in\mathcal{P}(N-2,N,\bar{\Omega})$ such that
\begin{equation}
\label{eq:ST2der}
\norm{u_{tt}(\bx,\bt)-p(\bx,\bt)}_\infty\le \frac{L}{(2N-2)^2}\ W\left(\frac{1}{(2N-2)^2}\right),\quad (\bx,\bt)\in\bar{\Omega},
\end{equation}
for some constant $L>0$ independent on $N$.

We define
\begin{equation}
\label{eq:util}
\tilde{u}(\bx,\bt)=u(\bx,-1)+u_t(\bx,-1)\ (\bt+1)+\int_{-1}^{\bt}\int_{-1}^\tau p(\bx,s)\,dsd\tau.
\end{equation}

We have
\begin{align*}
\left|u(\bx,\bt)-\tilde{u}(\bx,\bt)\right|&=\left|\int_{-1}^{\bt}\int_{-1}^\tau \left(u_{tt}(\bx,s)-p(\bx,s)\right)dsd\tau\right|\\
&\le \int_{-1}^{\bt}\int_{-1}^\tau \left|u_{tt}(\bx,s)-p(\bx,s)\right|dsd\tau\\
&\le \frac{L}{(2N-2)^2}\ W\left(\frac{1}{(2N-2)^2}\right)\int_{-1}^{\bt}\int_{-1}^\tau dsd\tau\\
&\le \frac{2L}{(2N-2)^2}\ W\left(\frac{1}{(2N-2)^2}\right)
\end{align*}
\end{proof}

\begin{lemma}
\label{lm:2}
Let $u_1,\,u_2\in\mathcal{C}^2_W(\bar{\Omega})$ satisfying equation~\eqref{eq:est0}. Then, there is a positive constant $L$ such that the following estimate holds
\begin{equation}
\label{eq:psiest}
\left|\mathcal{L}(u_1,\bt)-\mathcal{L}(u_2,\bt)\right|\le \frac{4L\beta}{(2N-2)^2}\ W\left(\frac{1}{(2N-2)^2}\right).
\end{equation}
\end{lemma}

\begin{proof}
Thanks to Lemma~\ref{lm:1} and the property of the peridynamic operator~\eqref{eq:lipschitz}, we find that there exists a positive constant $L$ such that 
\begin{align*}
\left|\mathcal{L}(u_1,\bt)-\mathcal{L}(u_2,\bt)\right|\le& L\int_{B_\delta(x)} C(x-x')\left|u_1(x',\bt)-u_2(x',\bt)\right|\,dx'\\
&+L\left|u_1(x,\bt)-u_2(x,\bt)\right|\int_{B_\delta(x)}C(x-x')\,dx'\\
\le&\,\frac{4L\beta}{(2N-2)^2}\ W\left(\frac{1}{(2N-2)^2}\right).
\end{align*}
\end{proof}

Now, we are able to prove that there exists at least one solution of~\eqref{eq:dismodel}.
\begin{theorem}
\label{th:existence}
Let $u\in\mathcal{C}^2_W(\bar{\Omega})$ be a solution of the peridynamic model~\eqref{eq:model}. Then there exists a positive integer $K$ such that for any $N\ge K$, the system~\eqref{eq:dismodel} admits a solution $\bar{u}^N=(\bar{u}^N_{nm})$ for $n,\,m=0,\dots,N$ such that
\begin{equation}
\label{eq:convest}
\left|u(\bx_k,\bt_h)-\bar{u}^N_{nm}\right|\le\frac{L}{(2N-2)^2}\ W\left(\frac{1}{(2N-2)^2}\right),\quad h,\,k=0,\dots,N
\end{equation}
for some positive constant $L$ independent of $N$.
\end{theorem}

\begin{proof}
We define
\begin{equation}
\label{eq:ahk}
\bar{u}^N_{nm}=\tilde{u}(x_n,t_m),\quad n,\,m=0,\dots,N,
\end{equation}
where $\tilde{u}$ is defined in~\eqref{eq:util} and satisfies equation~\eqref{eq:est0}.

By the definition of $\tilde{u}$, we find that it is a polynomial of degree at most $2N$. Thus, its second derivatives at Gauss-Lobatto nodes $(x_n,t_m)$, $n,\,m=0,\dots,N$ are given by
\begin{equation}
\label{eq:utila}
\tilde{u}_{tt}(x_n,t_m)=\sum_{j=0}^N\sum_{k=0}^N\sum_{\ell=0}^N \hat{D}_{k\ell}^{(t)} u_{j\ell} \ T_j(x_n)T_k(t_m).
\end{equation}

Using the relations~\eqref{eq:model},~\eqref{eq:est0},~\eqref{eq:psiest} and~\eqref{eq:utila}, we get
\begin{align}
\label{eq:est1}
\left|\sum_{j=0}^N\sum_{k=0}^N\sum_{\ell=0}^N \hat{D}_{k\ell}^{(t)} u_{j\ell} \ T_j(x_n)T_k(t_m)-\mathcal{L}(\bar{u}^N_{nm},t_m)\right|=\left|\tilde{u}_{tt}(x_n,t_m)-\mathcal{L}(\bar{u}^N_{nm},t_m)\right|&\\
\le\left|\tilde{u}_{tt}(x_n,t_m)-u_{tt}(x_n,t_m)\right|+\left|u_{tt}(x_n,t_m)-\mathcal{L}(\bar{u}^N_{nm},t_m)\right|&\notag\\
=\left|p(x_n,t_m)-u_{tt}(x_n,t_m)\right|+\left|\mathcal{L}(u(x_n,t_m))-\mathcal{L}(\bar{u}^N_{nm},t_m)\right|&\notag\\
\le\frac{L(1+4\beta)}{(2N-2)^2}\ W\left(\frac{1}{(2N-2)^2}\right)&.\notag
\end{align}

Moreover, we find an analogous estimate for the initial conditions:
\begin{align}
\label{eq:estinit1}
\left|\tilde{u}(x_n,-1)-u_0(x_n)\right|&\le\left|\tilde{u}(x_n,-1)-u(x_n,-1)\right|+\left|u(x_n,-1)-u_0(x_n)\right|\\
&\le\frac{2L}{(2N-2)^2}\ W\left(\frac{1}{(2N-2)^2}\right)\notag
\end{align}
and by equation~\eqref{eq:util}
\begin{align}
\label{eq:estinit2}
\left|\tilde{u}_t(x_n,-1)-v(x_n)\right|&\le\left|\tilde{u}_t(x_n,-1)-u_t(x_n,-1)\right|+\left|u_t(x_n,-1)-v(x_n)\right|\le0.
\end{align}

Therefore, if we choose $K$ such that
\[\max\{L(4\beta+1),2L\}\le\sqrt{N},\]
we have that $\bar{u}^N_{nm}$, $n,\,m=0,\dots,N$ defined in~\eqref{eq:ahk} satisfies~\eqref{eq:dismodel} for $N\ge K$, and this concludes the proof.
\end{proof}

Finally we prove that 
the solution of the system~\eqref{eq:dismodel} 
converges to the solution of the peridynamic model~\eqref{eq:model}.

\begin{theorem}
\label{th:convresult}
Let $\bar{u}^N=(\bar{u}^N_{nm})_{n,m=0}^N$, for $N\ge K$ be the sequence of solutions of~\eqref{eq:dismodel}, given by~\eqref{eq:ahk}, and let $u^N(\cdot,\cdot)$, for $N\ge K$ be its interpolating polynomial
\begin{equation}
\label{eq:interpol}
u^N(\bx,\bt)=\sum_{i=0}^N\sum_{j=0}^N \bar{a}^N_{ij} T_i(\bx)T_j(\bt),
\end{equation}
with
\[
\bar{a}^N_{ij}=\frac{1}{\gamma_i\gamma_j}\sum_{n=0}^N\sum_{m=0}^N \bar{u}^N_{nm}\ T_i(x_n)T_j(t_m) w_n w_m
\]
Let assume that for any $\bx\in[-1,1]$, the sequence $\{u^N(\bx,-1),u^N_t(\bx,-1),u^N_{tt}(\cdot,\cdot)\}_{N=K}^\infty$ has a subsequence $\{u^{N_i}(\bx,-1),u^{N_i}_t(\bx,-1),u^{N_i}_{tt}(\cdot,\cdot)\}_{i=0}^\infty$ uniformly converging to \\$(\vfi_1(\bx),\vfi_2(\bx),\vfi_3(\cdot,\cdot))$, where $\vfi_1$, $\vfi_2\in\mathcal{C}^2([-1,1])$ and $\vfi_3\in\mathcal{C}^2(\bar{\Omega})$. Then
\begin{equation}
\label{eq:lim}
\lim_{i\to\infty} u^{N_i}(\bx,\bt)=\tilde{u}(\bx,\bt),\quad (\bx,\bt)\in\bar{\Omega}
\end{equation}
is a solution of the peridynamic model~\eqref{eq:model}.
\end{theorem}

\begin{proof}
Due to our assumptions, we have
\begin{equation}
\label{eq:rel}
\tilde{u}(\bx,\bt)=\vfi_1(\bx)+\vfi_2(\bx)(\bt+1)+\int_{-1}^{\bt}\int_{-1}^\tau \vfi_3(\bx,s)\,dsd\tau.
\end{equation}
By contradiction, let assume that there is a $n\in\{1,\dots,N\}$ such that $\tilde{u}(\bx_n,\cdot)$ does not satisfy~\eqref{eq:model}. Hence, there is a $y\in(-1,1)$ such that
\[
\tilde{u}_{tt}(\bx_n,y)-\mathcal{L}(\tilde{u}(\bx_x,y),\bt_m)\ne0.
\]
Since the Gauss-Lobatto nodes $\{\bt_m\}_{m=0}^N$ are dense in $[-1,1]$ for $N\to \infty$, there is a subsequence $\{\bar{t}_{\ell_{N_i}}\}_{i=1}^\infty$ such that $\lim_{i\to\infty} \bt_{\ell_{N_i}}=y$ and $0<\ell_{N_i}<N_i$.

We have
\begin{align*}
0&\ne\tilde{u}_{tt}(\bx_n,y)-\mathcal{L}(\tilde{u}(\bx_n,y),\bt_m)\\
&=\lim_{i\to\infty}\left(\tilde{u}_{tt}(\bx_n,\bt_{\ell_{N_i}})-\mathcal{L}(\tilde{u}(\bx_n,\bt_{\ell_{N_i}}),\bt_m)\right)\\
&\le\lim_{i\to\infty}\frac{\sqrt{N_i}}{(2N_i-2)^2}\ W\left(\frac{1}{(2N_i-2)^2}\right)=0.
\end{align*}
Therefore, $\tilde{u}(\bx,\bt)$ satisfies the model~\eqref{eq:model} for all $\bt\in[-1,1]$ and $\bx=\bx_n$, $n=1,\dots,N$.

Using the same argument, we can prove that $\tilde{u}(\bx_n,-1)=u_0(\bx_n)$ and $\tilde{u}_t(\bx_n,-1)=v(\bx_n)$ for $n=0,\dots,N$ and this completes the proof.
\end{proof}

\section{Numerical tests}
\label{sec:test}

In what follows, we make some simulations to validate the proposed method and to study the properties of the solution of the peridynamic model~\eqref{eq:model}. All our codes have been written in MATLAB using an Intel(R) Core(TM) i7-5500U CPU @ 2.40GHz computer.

\subsection{Validation of the two-dimensional Chebyshev scheme}
\label{sec:validation}

The validation of the spectral Chebyshev method is made by comparing the obtained approximated solution with the solution of a benchmark problem.

We consider a bar on the spatial domain $[-1,1]$ and we let the solution evolve in the time interval $[-1,1]$, so that the computational domain is given by $\bar{\Omega}=[-1,1]\times[-1,1]$.

We fix $N>0$ and discretize $\bar{\Omega}$ by using the Gauss-Lobatto mesh points $(x_n,t_m)=(\cos(n\pi/N),\cos(m\pi/N))$, for $n,\,m=0,\dots,N$. We take $u_0(x)=e^{-x^2}$, $v(x)=0$ as initial conditions for $t=-1$, $\delta=0.1$ as the size of the horizon and $C(x)=e^{-x^2}$ as micromodulus function.

Figure~\ref{fig:evolution} depicts the evolution of the solution, computed by our method, corresponding to the initial condition $u_0(x)=e^{-x^2}$ on the domain $\bar{\Omega}$. To evaluate the convergence of the fully-discrete scheme we use the relative error $E^m$, defined as
\[
E^m=\frac{\sum_{n=0}^N \left|u^N_{nm}-u^\ast(x_n,t_m)\right|^2}{\sum_{n=0}^N \left|u^N_{nm}\right|^2},
\]
where $u^\ast$ is the reference solution.

Table~\ref{tab:error} shows the relative error $E^{m}$ between the exact and the numerical solution for different values of the total number of mesh points $N$ at time $t_m=1$. We find that the rate of convergence of the scheme is compatible with the theoretical result.

\begin{figure}[tbp]%
\centering
\includegraphics[width=0.6\textwidth]{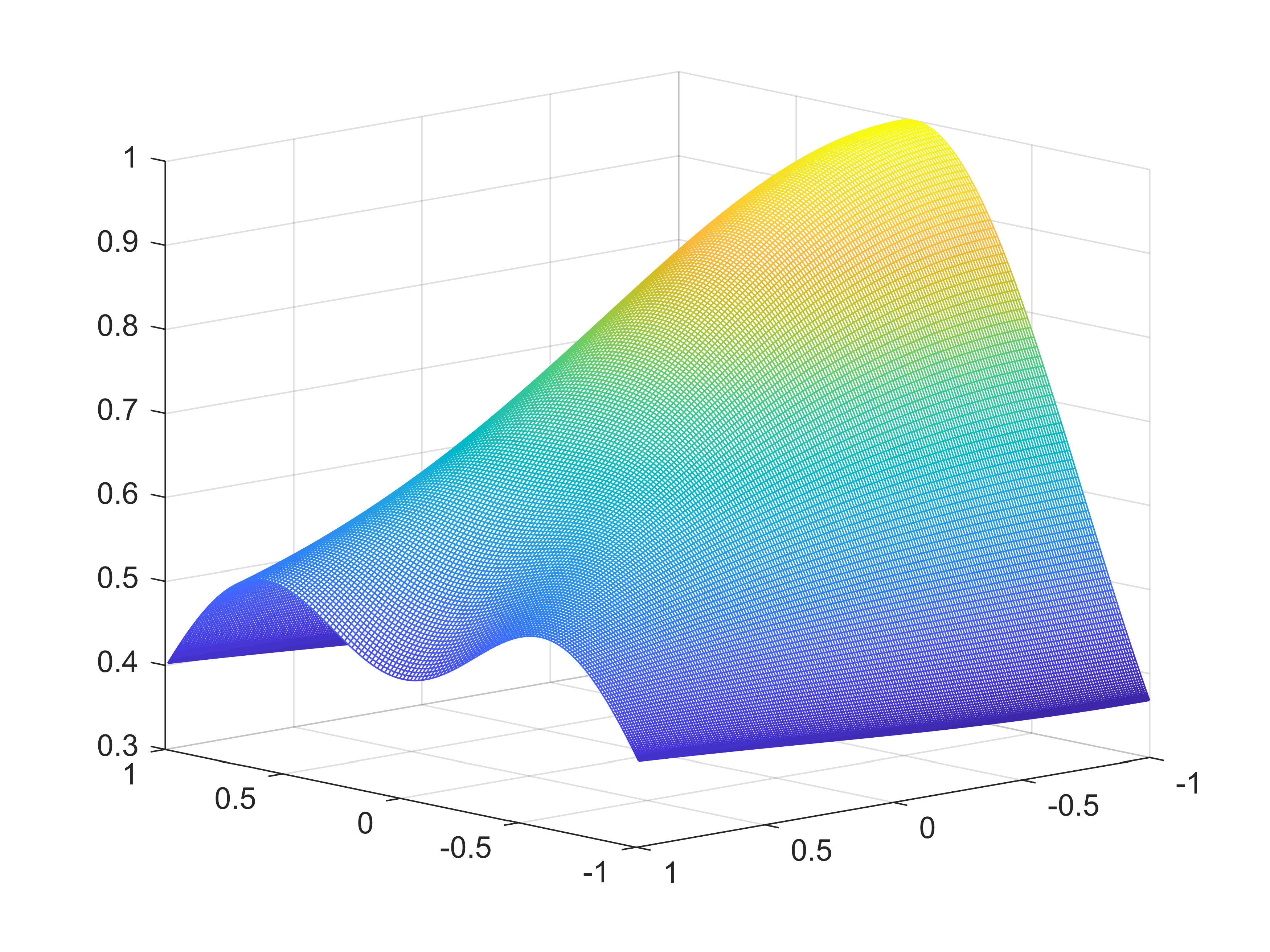}
\caption{With reference to Section~\ref{sec:validation}, the evolution of the solution of the problem. The parameters for the simulation are $\delta=0.1$, $N=1600$.}
\label{fig:evolution}
\end{figure}

\begin{table}%
\centering%
\renewcommand\arraystretch{1.3}
\begin{tabular}{ccc
}
\toprule
$N$& $E^{m}$ &convergence rate
\\
\midrule
$100$&$6.8594 \times 10^{-2}$&$-$
\\
$200$&$1.2508\times 10^{-2}$&$2.4552$
\\
$400$&$2.1895\times 10^{-3}$&$2.4847$
\\
$800$&$5.3782\times10^{-4}$&$2.3499$
\\
$1600$&$2.6481\times10^{-5}$&$2.7217$
\\
\bottomrule
\end{tabular}
\renewcommand\arraystretch{1}
\caption{With reference to Section~\ref{sec:validation}, the relative error between the exact solution and its numerical approximation, related to the initial condition $u_0(x)=e^{-x^2}$, at time $t_m=1$ as function of the number of discretization points.}
\label{tab:error}
\end{table}

Additionally, we analyze the performance of the method in terms of the computational cost required to complete the simulation.

We consider the same setting as before and we fix $u_0(x)=x/2$ as initial displacement. The solution of the problem is plotted in the left panel of Figure~\ref{fig:CPU}. In Table~\ref{tab:CPU} and in right panel of Figure~\ref{fig:CPU} we find that the method seems very competitive in terms of CPU cost. This is because the method exploits the properties of the Fast Fourier algorithm.

\begin{figure}
\centering
\begin{subfigure}[b]{.48\textwidth}
\includegraphics[width=\textwidth]{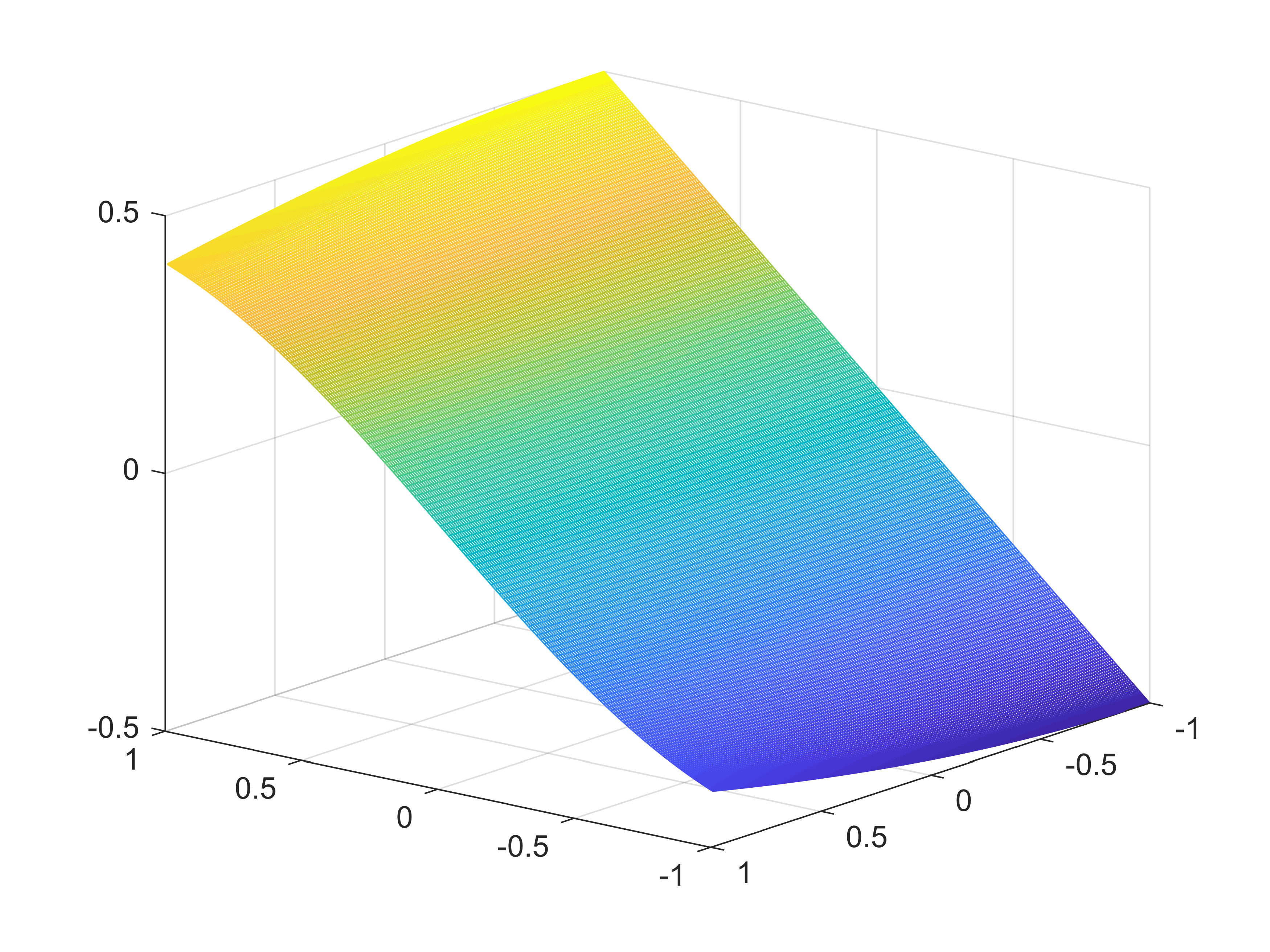}
\end{subfigure}
\begin{subfigure}[b]{.48\textwidth}
\includegraphics[width=\textwidth]{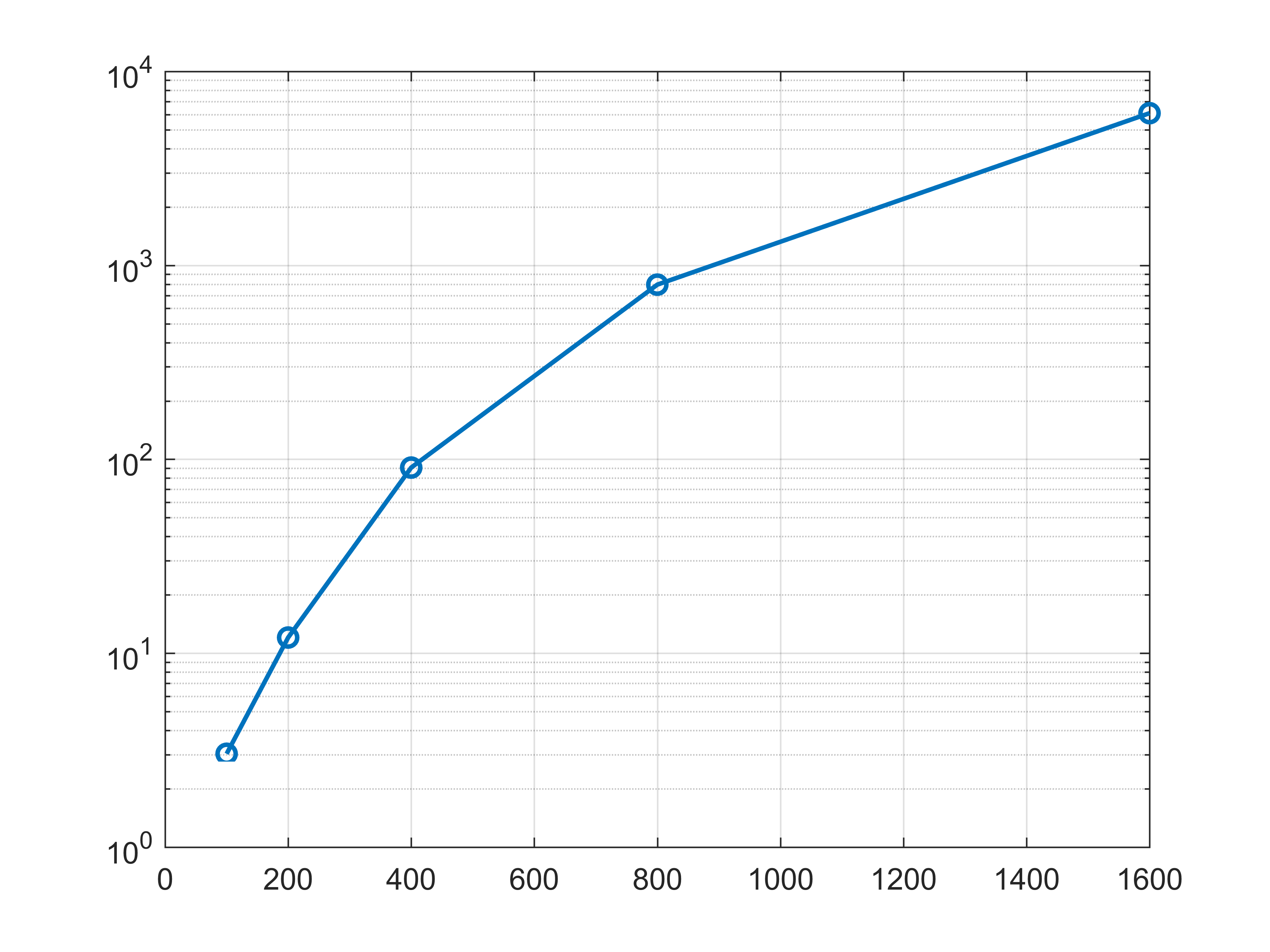}
\end{subfigure}
\caption{With reference to Section~\ref{sec:validation}. Left panel: the evolution of the solution corresponding the the initial displacement $u_0(x)=x/2$. Right panel: the behavior of the CPU cost depending of the total number of collocation points. For the simulation we fix $\delta=0.1$.}
\label{fig:CPU}
\end{figure}

\begin{table}%
\centering%

\renewcommand\arraystretch{1.3}
\begin{tabular}{cc
}
\toprule
$N$& CPU time [s]\\

\midrule
$100$&$3.0316 \times 10^{0}$
\\
$200$&$1.2051 \times 10^{1}$
\\
$400$&$9.0659\times 10^{1}$
\\
$800$&$7.9556\times 10^{2}$
\\
$1600$&$6.0924\times 10^{3}$
\\
\bottomrule
\end{tabular}
\renewcommand\arraystretch{1}
\caption{With reference to Section~\ref{sec:validation}, the execution time of the Chebyshev spectral method as function of the total number of collocation points $N$.}
\label{tab:CPU}
\end{table}

\subsection{A comparison between Chebyshev-Newmark-$\beta$ and the two-dimensional Chebyshev methods}
\label{sec:compare}

In~\cite{LPcheby}, the authors propose a spectral Chebyshev method for the spatial domain coupled with the Newmark-$\beta$ integrator to approximate the solution of the peridynamic model~\eqref{eq:model}. They showed good accuracy and performance in terms of CPU cost with respect to other spectral methods.

In this section, we make a comparison with our two-dimensional Chebyshev method and the Chebyshev-Newmark-$\beta$ method of~\cite{LPcheby}.

We clearly expect to find the same accuracy in space, as the spatial discretization method is practically the same. So, the aim of the comparison is to study the performance of the two methods in terms of CPU cost to complete the simulation.

We make some tests similar to those made in~\cite[Section 4.2]{LPcheby}. We work on $\bar{\Omega}=[-1,1]\times[-1,1]$, we take $u_0(x)=e^{-x^2}$ as initial displacement, $\delta=0.1$, $N=2000$ and $\beta=1/4$.

The solution of the problem at $t=1$ and its zoom on a small portion of the spatial domain is shown in Figure\ref{fig:zoom-compare}. As expected, we find a good agreement between the solution obtained with the two methods, and both of them are more accurate with respect to the penalized Fourier method. (We refer the reader to~\cite{LP} for a detailed description of the penalized Fourier spectral method).

\begin{figure}
\centering
\begin{subfigure}[b]{.48\textwidth}
\includegraphics[width=\textwidth]{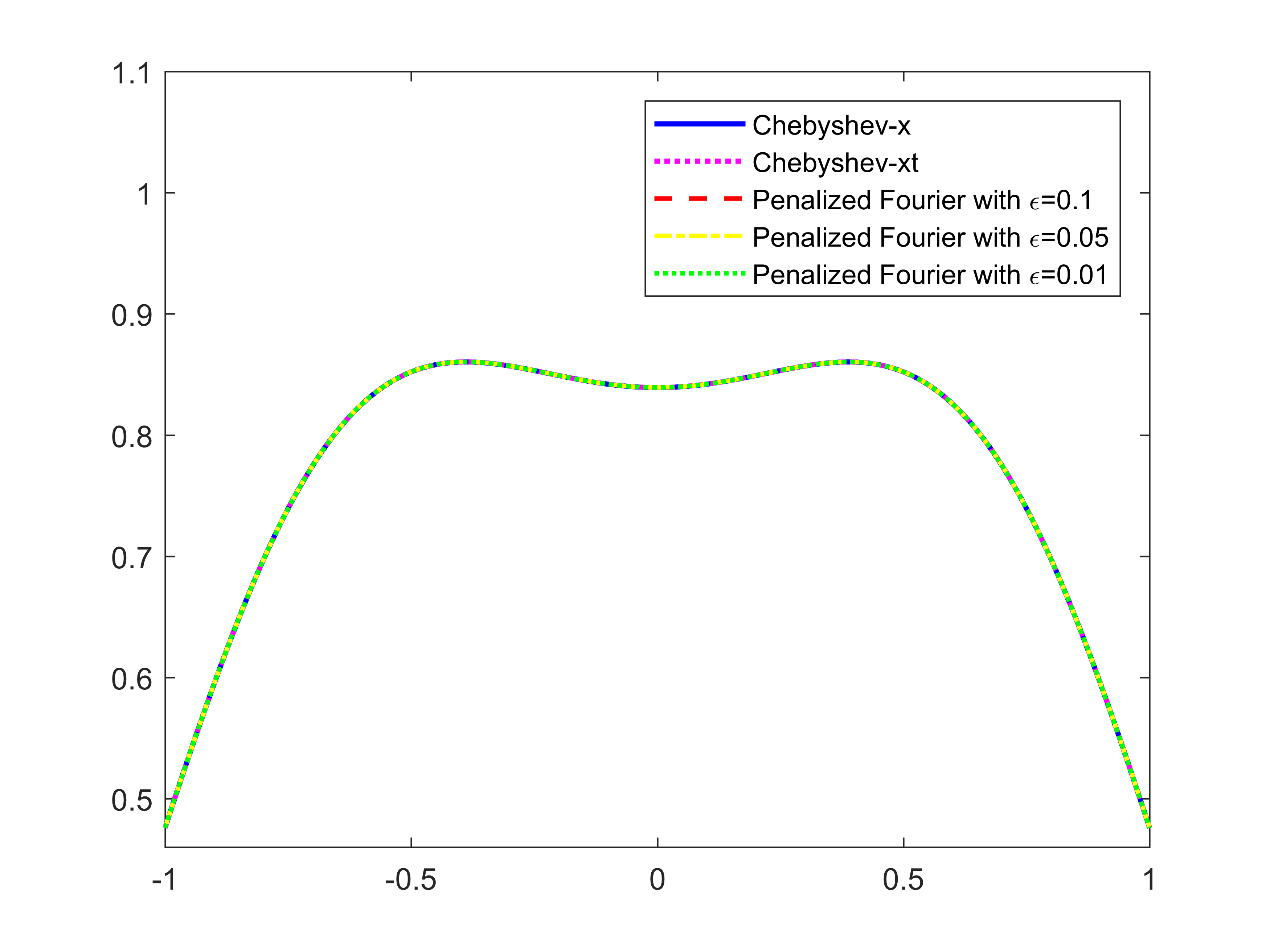}
\end{subfigure}
\begin{subfigure}[b]{.48\textwidth}
\includegraphics[width=\textwidth]{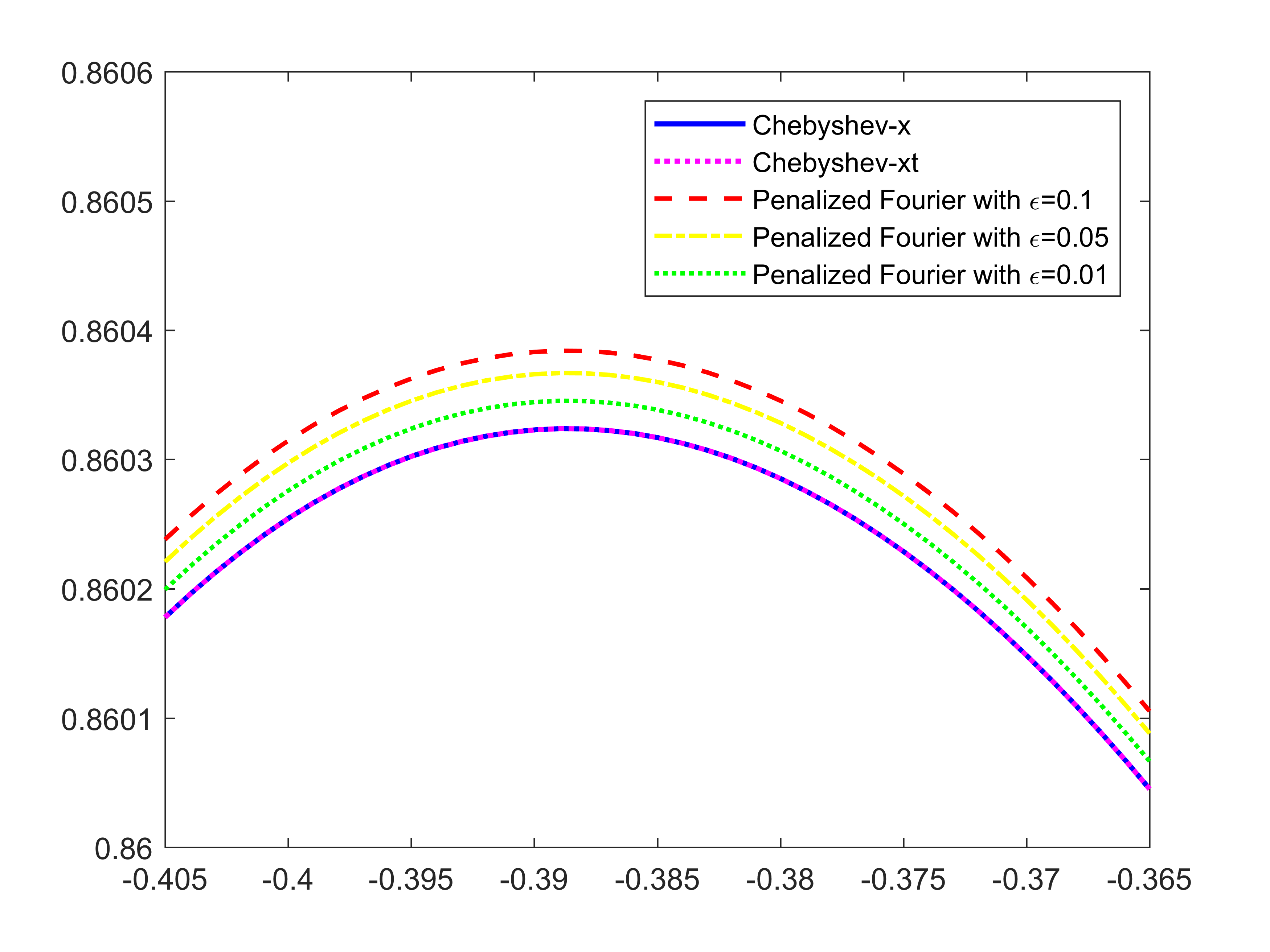}
\end{subfigure}
\caption{With reference to Section~\ref{sec:compare}. Left panel: the comparison between the solution obtained with different spectral methods at time $t=1$. Right panel: A zoom of the comparison in the spatial interval $[-0.405,-0.365]$. For the simulation we fix $\delta=0.1$, $N=2000$ and $\beta=1/4$.}
\label{fig:zoom-compare}
\end{figure}

Moreover, using the same setting, we analyze the methods in terms of time required to complete the simulation. When we deal with the Chebyshev-Newmark-$\beta$ method we have to vary both the space and the time step size, while this is implicitly done with the two-dimensional Chebyshev method if we fix $N>0$ as total number of grid points for space and time variables. In what follows, we fix $\Delta x=\Delta t= 2/N$. 

We summarize the results in Table~\ref{tab:CPUcompare}. We find a better result for the two-dimensional Chebyshev method, and this should depend on the fact that our method does not require any direct time integration as it is incorporated in the part of the algorithm in which we exploit the Fast Fourier Transform algorithm. Instead the computational cost to solve the algebraic system~\eqref{eq:modeldiscret} is compensated by the cost to solve the system derived by the implementation of the implicit Newmark-$\beta$ method.

\begin{table}%
\centering%
\renewcommand\arraystretch{1.3}
\begin{tabular}{ccc
}
\toprule
\multirow{2}*{$N$}& \multicolumn{2}{c}{CPU time [s]}  \\ \cmidrule(lr){2-3}
& Chebyshev-Newmark & 2D Chebyshev\\

\midrule
$128$&$7.8416 \times 10^{0}$&$3.6781 \times 10^{0}$
\\
$256$&$9.0142 \times 10^{1}$&$1.3984\times 10^{1}$
\\
$512$&$5.1257\times 10^{2}$&$1.0483\times 10^{2}$
\\
$1024$&$4.6573\times 10^{3}$&$9.1195\times 10^{2}$
\\
$2048$&$7.2561\times10^{4}$&$6.9763\times 10^{3}$
\\
\bottomrule
\end{tabular}
\renewcommand\arraystretch{1}
\caption{With reference to Section~\ref{sec:compare}, the execution time of the Chebyshev-Newmark-$\beta$ method and the two-dimensional Chebyshev method as function of the number of discretization points $N$, for $\beta=1/4$ and $\delta=0.1$.}
\label{tab:CPUcompare}
\end{table}

\subsection{The case of a discontinuous initial datum}
\label{sec:discont}

We now study the performance of Chebyshev spectral method applied to a problem with a discontinuous initial condition. We consider the same setting as in the previous sections and we take $u_0(x)=\chi_{[0,1]}(x)$ as initial displacement. We plot the dynamic of the solution in Figure~\ref{fig:discont}, while the error study is summarized in Table~\ref{tab:discont}.

We can notice the lost of one order of convergence due to the presence of a singularity in the initial displacement. This is in accordance with the results in~\cite{LPcheby,LP}.

\begin{figure}[tbp]%
\centering
\includegraphics[width=0.6\textwidth]{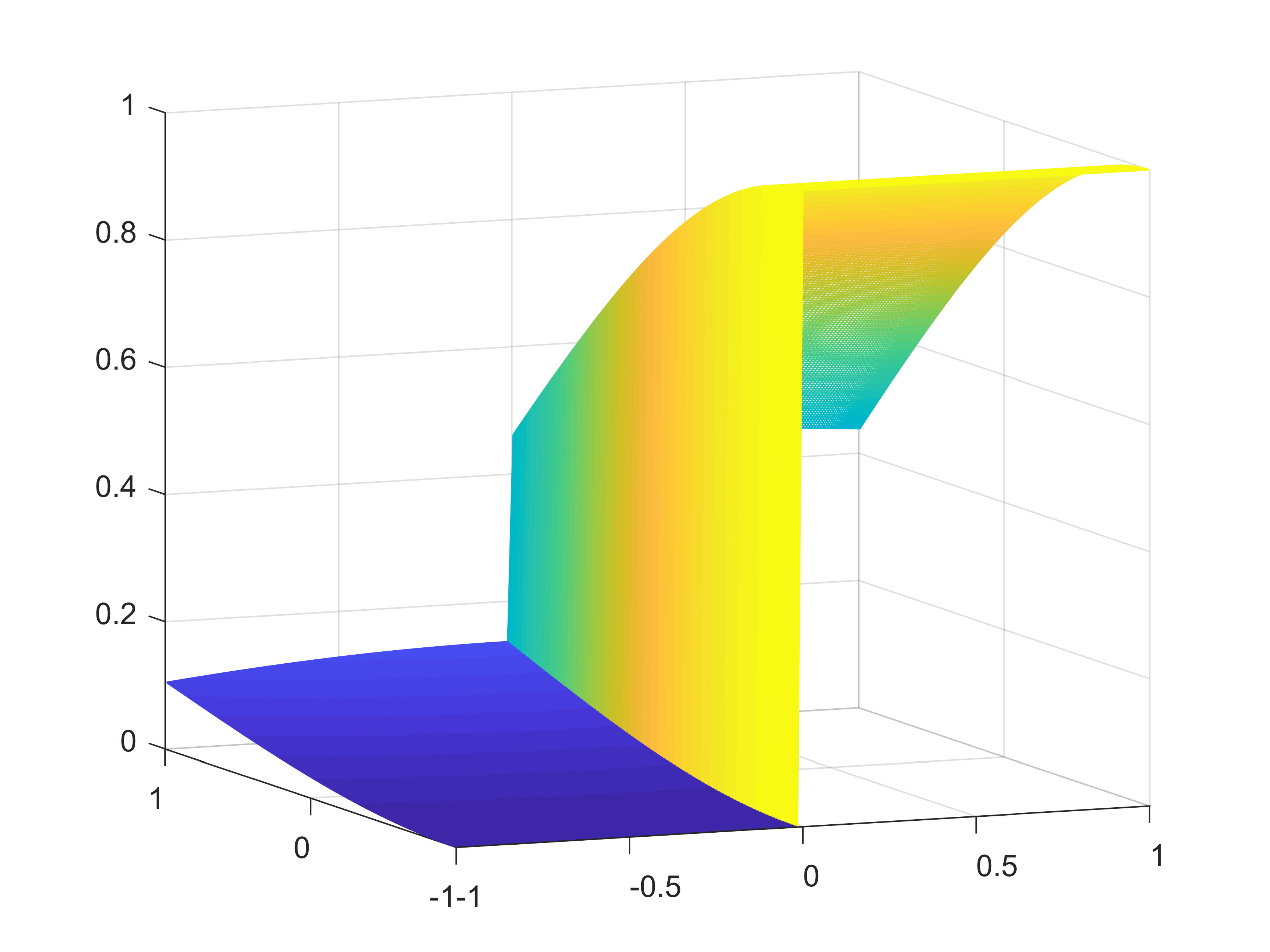}
\caption{With reference to Section~\ref{sec:discont}, the evolution of the solution corresponding to a singular initial displacement. The parameters for the simulation are $\delta=0.1$ and $N=1000$.}
\label{fig:discont}
\end{figure}

\begin{table}%
\centering%
\renewcommand\arraystretch{1.3}
\begin{tabular}{ccc
}
\toprule
$N$& $E^{m}$ &convergence rate
\\
\midrule
$100$&$4.1382 \times 10^{-1}$&$-$
\\
$200$&$1.4162\times 10^{-1}$&$1.5470$
\\
$400$&$6.9453\times 10^{-2}$&$1.2874$
\\
$800$&$3.6519\times10^{-2}$&$1.1535$
\\
$1600$&$2.6514\times10^{-3}$&$1.6528$
\\
\bottomrule
\end{tabular}
\renewcommand\arraystretch{1}
\caption{With reference to Section~\ref{sec:discont}, the relative error and the convergence rate, related to the initial displacement $u_0(x)=\chi_{[0,1]}(x)$, at time $t_m=1$ as function of the number of discretization points.}
\label{tab:discont}
\end{table}

\section{Conclusion and future works}
\label{sec:concl}

In this work, we propose a two-dimensional fast-convolution spectral method based on the implementation of Chebyshev polynomials to approximate the solution of a one-dimensional nonlinear peridynamic model having a power-type nonlinearity in the pairwise force function. The method results very accurate as it can exploit the benefits of the Fast Fourier Transform algorithm. Moreover, the idea to deal the problem in a bi-dimensional domain allows us to obtain the same accuracy in both space and time variables without requiring the implementation of a numerical scheme to integrate in time the discrete method. We prove the convergence of the proposed method and perform some simulations to validate the Chebyshev scheme and to study the properties of the solutions. 

In future, we plan to extend the method to higher dimensional problems and we aim to couple the approach to techniques based on mimetic and virtual element methods (see for example~\cite{Beirao_Lopez_Vacca_2017}).

\section{Declaration statement}
The authors consent for the publication.

The datasets used and/or analysed during the current study are available from the corresponding author on reasonable request.

The authors declare that they have no competing interests.

The authors declare that they gave their individual contributions in every section of the manuscript. All authors read and approved the final manuscript.

\section*{Acknowledgements}
This paper has been supported by GNCS of Istituto Nazionale di Alta Matematica, by PRIN 2017 ``Discontinuous dynamical systems: theory, numerics and applications'' and by Regione Puglia, ``Programma POR Puglia 2014/2020-Asse X-Azione 10.4 Research for Innovation-REFIN - (D1AB726C)''.


\bibliographystyle{plain}
\bibliography{biblioPeri2}

\end{document}